\documentclass{tpms-l}

\newtheorem{theorem}{Theorem}[section]

\theoremstyle{definition}
\newtheorem{definition}[theorem]{Definition}
\newtheorem{example}[theorem]{Example}

\theoremstyle{remark}

\newtheorem{Proposition}[theorem]{Proposition}

\newtheorem{assumption}[theorem]{Assumption}

\newtheorem{Remark}[theorem]{Remark}

\usepackage{graphicx}
\numberwithin{equation}{section}

\usepackage{hyperref}
\begin{document}

\title[Asymptotic behavior of random fractional Riesz-Bessel equations]{On asymptotic behavior of solutions to random fractional Riesz-Bessel equations with cyclic long memory initial conditions}

\author{Maha Mosaad A.~Alghamdi}
\address{Department of Mathematical and Physical Sciences, La Trobe University,
Melbourne, VIC 3086, Australia\newline
\indent Department of Mathematics, College of Science and Humanities,
Imam Abdulrahman Bin Faisal University, Jubail 31441, Saudi Arabia}
\email{mmghamdi@iau.edu.sa}

\author{Andriy Olenko}
\address{Department of Mathematical and Physical Sciences, La Trobe University,
Melbourne, VIC 3086, Australia}
\email{a.olenko@latrobe.edu.au}

\subjclass[2020]{Primary 60F05, 60H15; Secondary 60G15, 60G60}

\date{\today}

\keywords{Fractional Riesz-Bessel equations, Random partial differential equations, Seasonal long memory, Spectral singularities, Multiscaling limit theorems}

\begin{abstract}
This paper investigates fractional Riesz–Bessel equations with random initial conditions. The spectra of these random initial conditions exhibit singularities both at zero frequency and at non-zero frequencies, which correspond to the cases of classical long-range dependence and cyclic long-range dependence, respectively. Using spectral methods and asymptotic theory, it is shown that the rescaled solutions of the equations converge to spatio-temporal Gaussian random fields. The limit fields are stationary in space and non-stationary in time. The covariance and spectral structures of the resulting asymptotic random fields are provided. The paper further establishes multiscaling limit theorems for the case of regularly varying asymptotics. Numerical example illustrating the theoretical results is also presented.
\end{abstract}

\maketitle

\section{Introduction}
Fractional diffusion and fractional kinetic equations generalize the classical heat and diffusion equations by replacing first- and second-order derivatives with operators of fractional order. Fractional partial differential equations employ nonlocal operators that alow capturing memory effects and long-range interactions, see \cite{gorenflo2020mittag}. Such models often arise in the study of diffusion in porous media with fractal geometry, viscoelastic materials, complex media, seismic wave propagation, anomalous diffusion, and turbulence. Foundational contributions to this field were made by \cite{Caputo1967}, while extensive theoretical frameworks for fractional derivatives, integrals, and operators were subsequently developed, see \cite{Oldham1974} and \cite{podlubny1998fractional}.

Gay and Heyde \cite{gay1990class} demonstrated that fractional Laplace operators generate random fields with long-range dependence. Further analyzes by \cite{Angulo_Ruiz-Medina_Anh_Grecksch_2000} and \cite{Anh1999} focused on stochastic heat equations incorporating fractional Laplace operators. For random initial data exhibiting a singularity of the spectrum at the origin, which corresponds to long-range dependence, Anh and Leonenko \cite{Anh1999NonGaussianSF} proved that suitably rescaled solutions converge to non-Gaussian limits. These results established connections between fractional operators, spectral singularities, and limit theorems.

In this paper, we study fractional partial differential equations with random initial conditions. There is a substantial literature on this topic for non-fractional equations.

Random fields constructed by appropriately rescaling solutions of classical diffusion equations with random initial conditions, were studied by \cite{Albeverio1994} and \cite{Leonenko1998ScalingLO}. The Burgers equation with random forcing, which is connected to the heat equation through the Cole–Hopf transform, has been analyzed in this framework by  \cite{Leonenko1999LimitTF} and \cite{Leonenko1998ExactPA}. The publications \cite{ anh2021fractional,BKLO1,BKLO} studied a hyperbolic diffusion model on the unit sphere and obtained exact series representations for the solution of the Cauchy problem with random initial conditions. More recently, \cite{Broadbridge2024} considered the Cauchy problem for random diffusion in an expanding space–time framework. In addition to deriving probabilistic properties of the solutions, they quantified their extremal behavior by establishing upper bounds for the probabilities of large deviations. Furthermore, Beghin et al. \cite{beghin2000} investigated scaling regimes associated with the Airy and Korteweg–de Vries equations. Anh and Leonenko~\cite{ ANH2000239,Anh2002RenormalizationAH} developed renormalization and homogenization techniques for fractional-in-time and fractional-in-space diffusion equations with random inputs. Their results provided non-Gaussian central limit theorems and novel classes of limit random fields and motivated further studies of scaling laws and multiscaling behavior, see \cite{alghamdi2025multiscalingasymptoticbehaviorsolutions,  alghamdi2024, Anh1999NonGaussianSF, ANH2000239, leonenko2024fractional}.

Many results of the modern asymptotic theory use regular variation to generalise classical limit theorems. Regular varying functions, see \cite{Bingham_Goldie_Teugels, seneta2006regularly} provide wide, flexible deviations around power laws. Such results are also of great importance in statistics, as for modeling of real data, the exact mathematical power laws are very strict assumptions. For classical long-memory models, the corresponding results and methodology are relatively well developed, see, for example, \cite{anh2017rate,Leonenko2013SojournMO} and the references therein. However, for cyclic long-range dependent models, the singularity location usually changes when studying asymptotic changes, which requires new methods of investigation.

The paper presents several long-memory scenarios of random initial conditions for fractional Riesz–Bessel equations. The corresponding limit theorems are proven and discussed.

The article has the following structure.
Section \ref{defch5} introduces the main definitions and notation. Section \ref{sec3ch5} presents the main results concerning fractional Riesz–Bessel equations with random initial conditions. The spectra of these initial conditions exhibit singularities both at zero frequency and at nonzero frequencies, corresponding to classical long-range dependence and cyclic long-range dependence, respectively. Section \ref{sec4ch5} establishes multiscaling limit theorems for the case of regularly varying asymptotics. The paper also includes a numerical example illustrating the obtained results. Future research directions are discussed in Section \ref{conch5}.

All numerical computations and plotting in this paper were performed using the software Maple 2023. The corresponding Maple code is freely available in the folder
”Research materials” from the website \url{https://sites.google.com/site/olenkoandriy/}.

\section{ Definitions and notations} \label{defch5}
This section provides the main notations and background material required in the following sections.

Let $u(t,x)$ be a real-valued function of the two arguments $x \in \mathbb{R}^d$ and $t>0.$ The time derivative of order $\beta \in (0,1]$ is defined as follows:
\begin{equation*} 
\frac{\partial^{\beta} u(t,x)}{\partial t^{\beta}} =
\begin{cases}
\dfrac{\partial u}{\partial t}(t,x), & \text{if } \beta = 1, \\[10pt]
\left( \mathcal{D}_t^{\beta} u \right)(t,x), & \text{if } \beta \in (0,1),
\end{cases}
\end{equation*}
where
\[
\left( \mathcal{D}_t^{\beta} u \right)(t,x)
= \frac{1}{\Gamma(1 - \beta)}
\left[
\frac{\partial}{\partial t} \int_0^t (t - \tau)^{-\beta} u(\tau, x)\, d\tau
- \frac{u(0,x)}{t^{\beta}}
\right],
\qquad t>0,
\]
is the regularized fractional derivative in the Caputo–Djrbashian sense, see, for example,  \cite[(2.138)]{podlubny1998fractional}.

The publications \cite{Anh1999NonGaussianSF, ANH2000239} considered fractional Riesz–Bessel equations (FRBE)
\begin{equation}\label{modch5}
\frac{\partial^\beta u(t,x)}{\partial t^\beta} = -\mu (I - \Delta)^{\gamma/2} (-\Delta)^{\alpha/2} u(t,x), \quad t > 0,
\end{equation}
subject to random initial condition
\begin{equation}\label{modCch5}
u(0,x) = \xi(x),
\end{equation}
where $\alpha \geq 0, \, \gamma>0, \, \mu >0,$ $\xi(x)$,  $x \in \mathbb{R}^d$ is a measurable Gaussian random field on the probability space $(\Omega, \mathcal{F}, P)$ with mean zero, and $\Delta$ is the $d$-dimensional Laplace operator. The operators $-(I-\Delta)^{\gamma/2}$  and $(-\Delta)^{\alpha/2}$  are interpreted as the inverse operators of Bessel and Riesz potentials, respectively, see, for example, \cite{Anh1999NonGaussianSF, ANH2000239}. The random field $u(t,x)$ is interpreted as a mean-square solution of the initial value problem (\ref{modch5})-(\ref{modCch5}).

For the sake of simplicity, in this paper, we focus our attention on the one-dimensional case of $d=1.$ Then, $\Delta u(x)=\frac{\partial^{2}u(x)}{\partial x^2}$.
The covariance function of $\xi(x)$ can be written as
\begin{equation}\label{covch5}
 r(x)=\mathrm{Cov}(\xi(0), \xi(x)) = \int_{\mathbb{R}} e^{i  x \lambda } F(d\lambda),
\end{equation} where $F(\cdot)$ is the spectral measure.

If it can be represented as
\[F(\Delta)=\int_{\Delta} f_{\xi}(\lambda)d\lambda,\quad \Delta\in \mathcal{B}(\mathbb R), \] where $\mathcal{B}(\mathbb R)$ is $\sigma$ field of Borel sets of $\mathbb{R}$, then the function $f_{\xi}(\lambda), \lambda\in \mathbb R,$ which is integrable over $\mathbb R$, is called the spectral
density function of the statioray process $\xi(\cdot)$.
Then, the following spectral representation of the random process $\xi(x)$ holds true
\[
\xi(x) = \int_{\mathbb {R}}e^{i \lambda x } Z(d\lambda)=\int_{\mathbb {R}}e^{i \lambda x } \sqrt{f_{\xi}(\lambda)}W(d\lambda),
\]
 where $Z(\cdot)$ and $W(\cdot)$ denote the random measure and the white-noise random measure on $\mathbb{R}$, respectively. Note that, $\mathbb{E}|Z(d\lambda)|^2 = F(d\lambda).$

One can obtain the following spectral representation of the solution of the initial value problem (\ref{modch5})-(\ref{modCch5}), see \cite{ANH200377},
\begin{equation}\label{solution ch5}
    u(t,x) = \int_{\mathbb{R}} e^{i x \lambda} E_\beta(-\mu |\lambda|^{\alpha} (1+|\lambda|^2)^{\gamma/2} t^\beta) Z(d\lambda),
\end{equation}
where $E_\beta(\cdot)$ is the Mittag-Leffler function defined as
\[
E_\beta(s) := \sum_{k=0}^{\infty} \frac{s^k}{\Gamma(1 + \beta k)}, \qquad s \in \mathbb{R},\ 0 < \beta < 1.
\]
 For the negative values of its argument, it satisfies the inequality
\begin{equation}\label{upperch5}
   \frac{1}{1 + \Gamma(1-\beta)s} \leq E_\beta(-s) \leq\frac{1}{1 + \frac{s}{\Gamma(1+\beta)}}, \quad s\geq 0.
\end{equation}
For more details about the Mittag-Leffler function and its properties, see \cite{mainardi2014some, Simon2013ComparingFA}.

The covariance function of the solution field $u(t, x)$ is
\begin{align*}
    \mathrm{Cov}(u(t, x), u(t', x')) &= \int_{\mathbb{R}} e^{i
\lambda (x-x')} E_\beta(-\mu |\lambda|^{\alpha} (1+|\lambda|^2)^{\gamma/2} t^\beta)\\\nonumber&\times  E_\beta(-\mu |\lambda|^{\alpha} (1+|\lambda|^2)^{\gamma/2} (t')^\beta)  F(d\lambda).
\end{align*}

To introduce the class of random processes $\xi(x),$ $x \in \mathbb R,$ used as the initial condition~\eqref{modCch5}, the following assumption is used, see \cite{alghamdi2024} for more details.

\begin{assumption}\label{Asumptionmain}
    The covariance function (\ref{covch5}) has the form
\begin{align}\label{covariance of our problem11ch5}
    &r(x)=\sum_{j=0}^{n}\frac{\cos(w_{j}x)}{(1+x^{2})^{\kappa_j/2}}A_{j},\quad x\in\mathbb R,
\end{align}
where $\sum_{j=0}^{n}A_{j}=1, \, w_{0}=0,$ $ w_{j}>0,\,\kappa_j\in(0,1),\, j=1,...,n.
    $
\end{assumption}

The covariance function in \eqref{covariance of our problem11ch5} is non-integrable and has an oscillating behavior, which corresponds to the cyclic long-range dependence scenario.
It follows from (\ref{covariance of our problem11ch5}) that the corresponding spectral density has the representation
\begin{align*}
  &f(\lambda):=\sum_{j=1}^{n}\frac{c_1(\kappa_j)}{2}A_{j} \Big(K_{\frac{\kappa_j-1}{2}}\left(|\lambda+w_j| \right)|\lambda+w_j|^{\frac{\kappa_j-1}{2}}+ K_{\frac{\kappa_j-1}{2}}\left(|\lambda-w_j| \right)|\lambda-w_j|^{\frac{\kappa_j-1}{2}}\Big)\nonumber\\&+\frac{c_{1}(\kappa_0)}{2}A_{0}K_{\frac{\kappa_{0}-1}{2}}(|\lambda|)|\lambda|^{\frac{\kappa_{0}-1}{2}}=  \sum_{j=1}^{n}\frac{c_2(\kappa_j)}{2}A_j\left(\frac{1-\theta_{\kappa_j}(|\lambda+w_j|)}{|\lambda+w_j|^{1-\kappa_j}}+\frac{1-\theta_{\kappa_j}(|\lambda-w_j|)}{|\lambda-w_j|^{1-\kappa_j}}\right)\nonumber
  \end{align*}
  \begin{equation}\label{spectralch5}
 \hspace*{-4cm} +\frac{c_{2}(\kappa_0)}{2}A_{0}\frac{1-\theta_{\kappa_0(|\lambda|)}}{|\lambda|^{1-\kappa_0}}=\sum_{j=0}^{n}f_{\kappa_j,w_j}(\lambda),
  \end{equation}
where $c_1(\kappa_j):=2^{\mathbf{1}_{\{0\}}(j)}{2^{\frac{1-\kappa_j}{2}}}/\left({\sqrt{\pi}{\Gamma\left({\kappa_j}/{2} \right)}}\right)$, $c_{2}(\kappa_j):=\left(2^{1-\mathbf{1}_{\{0\}}(j)}\Gamma(\kappa_j)\cos\left({\kappa_{j}\pi}/{2}\right)\right)^{-1}$, $\mathbf{1}_{A}(\cdot)$ is the indicator function of a set $A,$ and $K_\nu(\cdot)$ is the modified Bessel function of the second kind  \[K_\nu(z)=\frac{1}{2}\int_{0}^{\infty}s^{\nu-1}\exp\left(-\frac{1}{2}\left(s+\frac{1}{s} \right)z\right)ds,\quad z\geq 0,\quad \nu\in \mathbb {R}.\]
The functions $\theta_{\kappa_j}(\cdot)$ are defined as
\begin{equation}\label{thetaj} \theta_{\kappa_j}(|\lambda|)=1-\frac{c_1(\kappa_j)}{c_2(\kappa_j)}K_{\frac{\kappa_j-1}{2}}(|\lambda|)|\lambda|^{\frac{1-\kappa_j}{2}}.\end{equation}
Let us introduce the values $w_{-j}=-w_{j},\, \kappa_{-j}=\kappa_{j},\, j=1,...,n.$ Then, the spectral density~\eqref{spectralch5} can be rewritten in the consise form
\begin{equation}\label{spectral density Ch5}
    f(\lambda)=\sum_{j=-n}^{n}\frac{c_{2}(\kappa_j)}{2}A_{j}\frac{1-\theta_{\kappa_j}(|\lambda+w_j|)}{|\lambda+w_j|^{1-\kappa_j}}.
\end{equation}

The spectral density has singularities at the origin and the points $\pm w_j$, $j=1, ..., n,$ with power-law types $|\lambda\pm w_{j}|^{\kappa_{j}-1}$.
For any $\kappa_j \in (0,1)$, the function $\theta_{\kappa_j}(\lambda)$ is bounded and satisfies $|\theta_{\kappa_j}(\lambda)| \leq 1$ for all $\lambda \in \mathbb{R}$.  Moreover, the following asymptotic holds
\begin{align*}
     \theta_{\kappa_j}(|\lambda|) = \frac{\Gamma\left(\frac{\kappa_j + 1}{2} \right)}{2^{1 - \kappa_j} \, \Gamma\left(\frac{3 - \kappa_j}{2} \right)} \left| \lambda \right|^{1 - \kappa_j}
- \frac{|\lambda|^2}{2(\kappa_j + 1)}  + o(|\lambda|^2), \quad |\lambda| \to 0.
\end{align*}
At infinity, the components $f_{\kappa_j,w_j}(\cdot)$ of the spectral density (\ref{spectralch5}) exhibit the following limit behavior
\begin{align*}
   f_{\kappa_j,w_j}(\lambda)\sim   \frac{A_j\sqrt{\pi}c_{1}(\kappa_j) (e^{w_j}+e^{-w_j})}{2^{1/2 +\mathbf{1}_{\{0\}}(j)}|\lambda|^{1-\kappa_j/2}}
e^{-|\lambda|}, \quad |\lambda| \to \infty.
\end{align*}

The spectral density $f(\cdot)$ is an even function. We assume that $W(\cdot)$ is a symmetric random measure. Hence, all random processes considered thereafter are real-valued.

\section{Multiscaling limits for fractional Riesz–Bessel equations}\label{sec3ch5}

This section examines the asymptotic behavior of FRBEs given by (\ref{modch5}) with initial random conditions defined by (\ref{modCch5}) for the cases of classical long-memory and cyclic long-memory behavior.

First, the result for only the cyclic long-memory case is presented.
\begin{theorem}\label{the1ch5} Consider the random field $u(t,x),$ $t>0,$ $x\in\mathbb R,$ defined by ~{\rm (\ref{modch5})} with  $\alpha>1/2$, and the random initial condition {\rm(\ref{modCch5})}. Let $\xi (x),$ $x\in \mathbb R,$ have a covariance function satisfying Assumption~{\rm\ref{Asumptionmain}} with $A_{0}=0$. If $\varepsilon\to 0,$ then the finite-dimensional distributions of the random fields
\begin{equation*}
    U_{\varepsilon}(t,x)={\varepsilon^{-\beta/2\alpha}}u\left(\frac{t}{\varepsilon},\frac{x}{\varepsilon^{\beta/\alpha}}\right)
\end{equation*}
converge weakly to the finite-dimensional distributions of the zero-mean Gaussian random field defined by
\begin{equation}\label{U_{01}ch5}
    U_{0}(t,x)=\sqrt{\sum_{j=1}^{n}c_1(\kappa_j)A_{j}K_{\frac{\kappa_j-1}{2}}(|w_j|)|w_j|^{\frac{\kappa_j-1}{2}}}\int_{\mathbb R}e^{i\lambda x} E_{\beta}\left(-\mu t^{\beta}|\lambda|^{\alpha}\right)W(d\lambda),
\end{equation}
with the covariance function
\begin{align}
  {\rm Cov}\left(U_{0}(t,x), U_{0}(t',x')\right)&=\sum_{j=1}^{n}2c_1(\kappa_j)A_{j}K_{\frac{\kappa_j-1}{2}}(|w_j|)|w_j|^{\frac{\kappa_j-1}{2}}\nonumber\\
  &\times\int_{0}^{
  +\infty}\cos\lambda(x-x')E_{\beta}(-\mu t^{\beta}|\lambda|^{\alpha})E_{\beta}(-\mu (t')^{\beta}|\lambda|^{\alpha})d\lambda.\label{cov of Bessel ch5}
\end{align}
\end{theorem}
\begin{proof}
    The proof combines the methods from  \cite{alghamdi2025multiscalingasymptoticbehaviorsolutions} and \cite{alghamdi2024}. By \eqref{solution ch5}, it holds in the sense of the finite-dimensional distributions that
    \begin{equation*}
        U_{\varepsilon}(t,x)=\frac{1}{\varepsilon^{\beta/2\alpha}}\int_{\mathbb R}e^{i\lambda x/\varepsilon^{\beta/\alpha}}E_{\beta}\left(-\mu |\lambda|^{\alpha}(1+|\lambda|^{2})^{\frac{\gamma}{2}}\left(\frac{t}{\varepsilon} \right)^{\beta} \right)\sqrt{f(\lambda)}W(d\lambda).
    \end{equation*}
    Note, that by \eqref{thetaj}
     \begin{equation*}
       c_{1}(\kappa_{j})K_{\frac{\kappa_{j}-1}{2}}(|\lambda|)|\lambda|^{\frac{\kappa_{j}-1}{2}}=c_2(\kappa_j)\frac{1-{\theta_{\kappa_j}(|\lambda|)}}{|\lambda|^{1-\kappa_j}}, \, \lambda\in \mathbb{R}.
   \end{equation*}
    By \eqref{spectralch5}, the change of variable $\lambda=\tilde{\lambda}\varepsilon^{\beta/\alpha}$, and using the Brownian scaling property, \\$W(d(\tilde{\lambda}\varepsilon^{\beta/\alpha}))\overset{d}{=}\varepsilon^{\beta/2\alpha}W(d\tilde{\lambda}),$ one obtains
    \begin{align}\label{U1}
        U_{\varepsilon}(t,x)&\overset{d}{=}\frac{1}{\varepsilon^{\beta/2\alpha}}\int_{\mathbb R}e^{i\tilde{\lambda}x} E_{\beta}\left(-\mu |\tilde{\lambda}\varepsilon^{\beta/\alpha}|^{\alpha}(1+|\tilde{\lambda}\varepsilon^{\beta/\alpha}|^{2})^{\gamma/2}\left(\frac{t}{\varepsilon} \right)^{\beta}\right)\nonumber\\
        &\times\sqrt{f(\tilde{\lambda}\varepsilon^{\beta/\alpha})}W(d\tilde{\lambda}\varepsilon^{\beta/\alpha})\overset{d}{=}\int_{\mathbb R}e^{i\tilde{\lambda}x}E_{\beta}\left(-\mu |\tilde{\lambda}|^{\alpha}(1+|\tilde{\lambda}\varepsilon^{\beta/\alpha}|^{2})^{\gamma/2}t^{\beta}\right)\nonumber\\
        &\times \sqrt{\sum_{j=-n}^{n}A_{j}\frac{c_{2}(\kappa_j)}{2}\frac{1-\theta_{\kappa_j}(|\tilde{\lambda}\varepsilon^{\beta/\alpha}+w_j|)}{|\tilde{\lambda}\varepsilon^{\beta/\alpha}+w_j|^{1-\kappa_j}}}W(d\tilde{\lambda}).
    \end{align}

Then, by   (\ref{U_{01}ch5}) and (\ref{U1})
\begin{align}\label{difch5}
    R(t,x)&=\mathbb E(U_{\varepsilon}(t,x)-U_{0}(t,x))^{2}=\mathbb E\Bigg(\int_{\mathbb R}e^{i{\lambda}x}E_{\beta}\left(-\mu |{\lambda}|^{\alpha}(1+|{\lambda}\varepsilon^{\beta/\alpha}|^{2})^{\gamma/2}t^{\beta}\right)\nonumber\\
        &\times \sqrt{\sum_{j=-n}^{n}A_{j}\frac{c_{2}(\kappa_j)}{2}\frac{1-\theta_{\kappa_j}(|{\lambda}\varepsilon^{\beta/\alpha}+w_j|)}{|{\lambda}\varepsilon^{\beta/\alpha}+w_j|^{1-\kappa_j}}}W(d{\lambda})\nonumber\\&- \sqrt{\sum_{j=1}^{n}A_{j}\frac{c_{2}(\kappa_{j})}{|w_{j}|^{1-\kappa_{j}}}(1-\theta_{\kappa_j}(|w_j|))}\int_{\mathbb R}e^{i\lambda x} E_{\beta}\left(-\mu t^{\beta}|\lambda|^{\alpha}\right)W(d\lambda)\Bigg)^{2}\nonumber
        \\&=\int_{\mathbb R}E^{2}_{\beta}(-\mu t^{\beta}|\lambda|^{\alpha})\left( {Q}_{\varepsilon^{\beta/\alpha}}(\lambda)-\sqrt{\sum_{j=1}^{n}A_{j}\frac{c_{2}(\kappa_{j})}{|w_{j}|^{1-\kappa_{j}}}(1-\theta_{\kappa_j}(|w_j|))}\right)^{2}d\lambda,
\end{align}
    where
\begin{align*}
    {Q}_{\varepsilon^{\beta/\alpha}}&(\lambda):=\frac{ E_{\beta}\left(-\mu|\lambda |^{\alpha}(1+|\lambda\varepsilon^{\beta/\alpha}|^{2})^{\gamma/2}t^{\beta}\right)}{ E_{\beta}(-\mu t^{\beta}|\lambda|^{\alpha})}\sqrt{\sum_{j=-n}^{n}A_{j}\frac{c_{2}(\kappa_j)}{2}\frac{1-\theta_{\kappa_j}(|\lambda\varepsilon^{\beta/\alpha}+w_j|)}{|\lambda\varepsilon^{\beta/\alpha}+w_j|^{1-\kappa_j}}}.
\end{align*}
Note that for each fixed $\lambda$
\begin{equation}\label{Qch5}
    \lim_{\varepsilon\to 0} {Q}_{\varepsilon^{\beta/\alpha}}(\lambda)=\sqrt{\sum_{j=1}^{n}A_{j}\frac{c_{2}(\kappa_{j})}{|w_{j}|^{1-\kappa_{j}}}(1-\theta_{\kappa_j}(|w_j|))}=f^{1/2}(0).
\end{equation}
Therefore, the integrand in \eqref{difch5} converges pointwise to zero. As it holds \[\left({Q}_{\varepsilon^{\beta/\alpha}}(\lambda)-\sqrt{\sum_{j=1}^{n}\frac{A_{j}c_{2}(\kappa_{j})}{|w_{j}|^{1-\kappa_{j}}}(1-\theta_{\kappa_j}(|w_j|))}\right)^{2} \leq {Q}^{2}_{\varepsilon^{\beta/\alpha}}(\lambda)+\sum_{j=1}^{n}\frac{A_{j}c_{2}(\kappa_{j})}{|w_{j}|^{1-\kappa_{j}}}(1-\theta_{\kappa_j}(|w_j|)),\]  one can apply the generalized Lebesgue's dominated convergence theorem.

To justify its conditions, one needs to show that the limit
\begin{equation}\label{justify condition on Besselch5}
    \begin{split}
        & \lim_{\varepsilon\to 0}\int_{\mathbb R} E_{\beta}^{2}(-\mu t^{\beta}|\lambda|^{\alpha})\left(Q^2_{\varepsilon^{\beta/\alpha}}(\lambda)+\sum_{j=1}^{n}A_{j}\frac{c_{2}(\kappa_{j})}{|w_{j}|^{1-\kappa_{j}}}(1-\theta_{\kappa_j}(|w_j|))\right)d\lambda\\&=\int_{\mathbb R} E_{\beta}^{2}(-\mu t^{\beta}|\lambda|^{\alpha})\lim_{\varepsilon\to 0}\left(Q^2_{\varepsilon^{\beta/\alpha}}(\lambda)+\sum_{j=1}^{n}A_{j}\frac{c_{2}(\kappa_{j})}{|w_{j}|^{1-\kappa_{j}}}(1-\theta_{\kappa_j}(|w_j|))\right)d\lambda< \infty.
    \end{split}
\end{equation}

The boundedness of the last term in (\ref{justify condition on Besselch5}) follows from \eqref{Qch5} and the boundedness of the spectral density at zero
\begin{align}\label{finite1ch5}
         &\int_{\mathbb R} E_{\beta}^{2}(-\mu t^{\beta}|\lambda|^{\alpha})\lim_{\varepsilon\to 0}\left(Q^2_{\varepsilon^{\beta/\alpha}}(\lambda)+\sum_{j=1}^{n}A_{j}\frac{c_{2}(\kappa_{j})}{|w_{j}|^{1-\kappa_{j}}}(1-\theta_{\kappa_j}(|w_j|))\right)d\lambda\nonumber\\&=2f(0)\int_{\mathbb R} E_{\beta}^{2}(-\mu t^{\beta}|\lambda|^{\alpha})d\lambda<\infty.
\end{align}

The integral in (\ref{finite1ch5}) is finite if $\alpha>1/2$. The condition $\alpha>1/2$ is sufficient and necessary for the boundedness of the integral, see the proof of \cite[equation (30)]{alghamdi2024}.

The first integral in (\ref{justify condition on Besselch5}) can be written as
\begin{align}\label{3Ich5}&
\int_{\mathbb{R}} E_{\beta}^{2}\!\bigl(-\mu t^{\beta}|\lambda|^{\alpha}\bigr)\Bigl(Q^{2}_{\varepsilon^{\beta/\alpha}}(\lambda)+\sum_{j=1}^{n} A_j \frac{c_{2}(\kappa_j)}{|w_j|^{1-\kappa_j}}\bigl(1-\theta_{\kappa_j}(|w_j|)\bigr)\Bigr)\,d\lambda \nonumber\\
&=
\int_{\mathbb{R}}
E^{2}_{\beta}\!\left(-\mu|\lambda|^{\alpha}\bigl(1+|\lambda\varepsilon^{\beta/\alpha}|^{2}\bigr)^{\gamma/2}t^{\beta}\right)
\sum_{j=-n}^{n}A_j\frac{c_2(\kappa_j)(1-\theta_{\kappa_j}(|\lambda\varepsilon^{\beta/\alpha}+w_j|))}{2|\lambda\varepsilon^{\beta/\alpha}+w_j|^{1-\kappa_j}}\,d\lambda\nonumber\\&
+\int_{\mathbb{R}} E_{\beta}^{2}\!\bigl(-\mu t^{\beta}|\lambda|^{\alpha}\bigl)\sum_{j=1}^{n}A_{j} c_{2}(\kappa_{j})\frac{1-{\theta_{\kappa_j}}(|w_{j}|)}{|w_{j}|^{1-\kappa_{j}}}\,d\lambda
 \nonumber\\
&=:
I_{1}(\varepsilon^{\beta/\alpha})+\int_{\mathbb{R}} E_{\beta}^{2}\!\bigl(-\mu t^{\beta}|\lambda|^{\alpha}\bigl)\sum_{j=1}^{n}A_{j} c_{2}(\kappa_{j})\frac{1-{\theta_{\kappa_j}}(|w_{j}|)}{|w_{j}|^{1-\kappa_{j}}}\,d\lambda.
\end{align}
Let us split the integration in $I_{1}(\varepsilon^{\beta/\alpha})$  into two regions, $|\lambda|\leq \varepsilon^{-\delta}$ and $|\lambda|> \varepsilon^{-\delta},$  where $\delta \in (\beta/(2\alpha^2),\beta/\alpha),$
\begin{align}\label{firstintegralch5}
  &
\int_{|\lambda|\leq \varepsilon^{-\delta}}
E^{2}_{\beta}\!\left(-\mu|\lambda|^{\alpha}\bigl(1+|\lambda\varepsilon^{\beta/\alpha}|^{2}\bigr)^{\gamma/2}t^{\beta}\right)
\sum_{j=-n}^{n}A_j\frac{c_2(\kappa_j)(1-\theta_{\kappa_j}(|\lambda\varepsilon^{\beta/\alpha}+w_j|))}{2|\lambda\varepsilon^{\beta/\alpha}+w_j|^{1-\kappa_j}}\,d\lambda\nonumber\\&+\int_{|\lambda|> \varepsilon^{-\delta}}
E^{2}_{\beta}\!\left(-\mu|\lambda|^{\alpha}\bigl(1+|\lambda\varepsilon^{\beta/\alpha}|^{2}\bigr)^{\gamma/2}t^{\beta}\right)
\sum_{j=-n}^{n}A_j\frac{c_2(\kappa_j)(1-\theta_{\kappa_j}(|\lambda\varepsilon^{\beta/\alpha}+w_j|))}{2|\lambda\varepsilon^{\beta/\alpha}+w_j|^{1-\kappa_j}}\,d\lambda.
\end{align}
Notice that $f(\lambda)$ is an even function, which increases for positive $\lambda$ in a neighbourhood of $0.$
By the complete monotonicity decreasing property of the Mittag-Leffler functions of negative arguments for $\beta\in(0,1]$, see \cite[Proposition 3.10]{gorenflo2020mittag},  one obtains lower and upper bounds for the first integral in (\ref{firstintegralch5})
\begin{equation}\label{lboundch5}
\sum_{j=-n}^{n} A_j \frac{c_{2}(\kappa_j)(1-\theta_{\kappa_j}(|w_j|))}{2|w_j|^{1-\kappa_j}}
\int_{|\lambda|\leq \varepsilon^{-\delta}} E^{2}_{\beta}(-\mu |\lambda|^{\alpha}(1+|\lambda\varepsilon^{\beta/\alpha}|^{2})^{\gamma/2}t^{\beta} )d\lambda
\end{equation}
\begin{align}\le & \int_{|\lambda|\leq \varepsilon^{-\delta}} E^{2}_{\beta}(-\mu |\lambda|^{\alpha}(1+|\lambda\varepsilon^{\beta/\alpha}|^{2})^{\gamma/2}t^{\beta} )\sum_{j=-n}^{n} A_j\frac{c_{2}(\kappa_j)(1-\theta_{\kappa_j}(|\lambda\varepsilon^{\beta/\alpha}+w_j|))}{2|\lambda\varepsilon^{\beta/\alpha}+w_j|^{1-\kappa_j}}d\lambda \nonumber\\
 \le&\int_{|\lambda|\leq \varepsilon^{-\delta}} E^{2}_{\beta}(-\mu |\lambda|^{\alpha}t^{\beta} )\sum_{j=-n}^{n} A_j\frac{c_{2}(\kappa_j)(1-\theta_{\kappa_j}(|\lambda\varepsilon^{\beta/\alpha}+w_j|))}{2|\lambda\varepsilon^{\beta/\alpha}+w_j|^{1-\kappa_j}}d\lambda \nonumber\\
 \le& \sum_{j=-n}^{n} A_j
\frac{c_{2}(\kappa_j)(1-\theta_{\kappa_j}(|\varepsilon^{\beta/\alpha-\delta}+w_j|))}{2|\varepsilon^{\beta/\alpha-\delta}+w_j|^{1-\kappa_j}}\int_{|\lambda|\leq \varepsilon^{-\delta}} E_{\beta}^{2}(-\mu |\lambda|^{\alpha}t^{\beta} )d\lambda\nonumber.
\end{align}

Then,
\begin{align}\label{1Ifirst ch5}
    \lim_{\varepsilon\to 0}&\sum_{j=-n}^{n} A_j \frac{c_{2}(\kappa_j)(1-\theta_{\kappa_j}(|\varepsilon^{\beta/\alpha-\delta}
    +w_j|))}{2|\varepsilon^{\beta/\alpha-\delta}+w_j|^{1-\kappa_j}}\int_{|\lambda|\leq \varepsilon^{-\delta}} E_{\beta}^{2}(-\mu |\lambda|^{\alpha}t^{\beta} )d\lambda\nonumber\\&=\sum_{j=1}^{n} A_j \frac{c_{2}(\kappa_j)(1-\theta_{\kappa_j}(|w_j|))}{|w|^{1-\kappa_j}}\int_{\mathbb R} E_{\beta}^{2}(-\mu |\lambda|^{\alpha}t^{\beta} )d\lambda,
\end{align}
where, as was proved in (\ref{finite1ch5}),  the last integral is finite.

Hence, the first integral in (\ref{firstintegralch5}) is uniformly bounded. Noting that the integrands in~(\ref{lboundch5}) are bounded by $E_{\beta}^{2}(-\mu |\lambda|^{\alpha}t^{\beta} )$ and pointwise converge to this bound, by the dominated convergence theorem,  one obtains that the lower bound converges to the same values as the upper one.
Therefore, the first integral in (\ref{firstintegralch5}) converges to the value $\sum_{j=1}^{n} c_{2}(\kappa_j)A_j (1-\theta_{\kappa_j}(|w_j|))|w_j|^{\kappa_j-1}\int_{\mathbb R} E_{\beta}^{2}(-\mu |\lambda|^{\alpha}t^{\beta} )d\lambda.$

The second integral in (\ref{firstintegralch5}) can be estimated as
\begin{align}
    0&<\int_{|\lambda|> \varepsilon^{-\delta}}E^{2}_{\beta}(-\mu |\lambda|^{\alpha}(1+|\lambda\varepsilon^{\beta/\alpha}|^{2})^{\gamma/2}t^{\beta} )\sum_{j=-n}^{n} A_j \frac{c_{2}(\kappa_j)(1-\theta_{\kappa_j}(|\lambda\varepsilon^{\beta/\alpha}+w_j|))}{2|\lambda\varepsilon^{\beta/\alpha}+w_j|^{1-\kappa_j}}d\lambda\nonumber\\&\leq \frac{ E_{\beta}^{2}(-\mu \varepsilon^{-\delta\alpha}(1+|\varepsilon^{{\beta/\alpha}-\delta}|^{2})^{\gamma/2}{t}^{\beta}) }{ \varepsilon^{\beta/\alpha}}\int_{|\tilde{\lambda}|>\varepsilon^{{\beta/\alpha}-\delta}}\sum_{j=-n}^{n} A_j \frac{c_{2}(\kappa_j)(1-\theta_{\kappa_j}(|\tilde{\lambda}+w_j|))}{2|\tilde{\lambda}+w_j|^{1-\kappa_j}}d\tilde{\lambda}\nonumber
    \end{align}
    \begin{equation}\label{secondintegralch5}
        \leq  \frac{E_{\beta}^{2}(-\mu {t}^{\beta}\varepsilon^{-\delta\alpha}) }{ \varepsilon^{\beta/\alpha}}\int_{\mathbb R}\sum_{j=-n}^{n} A_j\frac{c_{2}(\kappa_j)(1-\theta_{\kappa_j}(|\tilde{\lambda}+w_j|))}{2|\tilde{\lambda}+w_j|^{1-\kappa_j}}d\tilde{\lambda}\to 0, \ \mbox{when}\ \varepsilon\to 0,
    \end{equation}
     as it follows from the upper bound in (\ref{upperch5}) and the condition $\delta \in (\beta/(2\alpha^2),\beta/\alpha)$ that
 \[ \frac{E_{\beta}^{2}(-\mu {t}^{\beta}\varepsilon^{-\delta\alpha}) }{ \varepsilon^{\beta/\alpha}}\le \frac{1}{ \varepsilon^{\beta/\alpha}+\mu^{2} {t}^{2\beta}\varepsilon^{\beta/\alpha-2\delta\alpha}/\Gamma^{2}(1+\beta) }\to 0, \quad \mbox{when}\ \varepsilon\to 0,  \]
and due to the integrability of the spectral density, the integral in (\ref{secondintegralch5}) is finite.

Hence, (\rm\ref{justify condition on Besselch5}) follows from the results (\rm\ref{3Ich5}), (\rm\ref{firstintegralch5}),  (\rm\ref{1Ifirst ch5}) and (\ref{secondintegralch5}).
 Thus, by the generalized dominated converges theorem $\lim_{\varepsilon \to 0}E{\vert U_{\varepsilon}(t,x)-U_{0}(t,x)\vert}^{2}=0,$ which implies the convergence of finite dimensional distributions.

 Finally, equation \eqref{cov of Bessel ch5} follows from \eqref{U_{01}ch5}, the isometry property of Wiener-It\^o stochastic integrals, and
 \begin{align}
     &{\rm Cov}(U_{0}(t,x),U_{0}(t',x'))=\mathbb E U_{0}(t,x)U_{\varepsilon}(t^{\prime},x^{\prime})\nonumber\\&=\sum_{j=1}^{n}c_1(\kappa_j)A_{j}K_{\frac{\kappa_j-1}{2}}(|w_j|)|w_j|^{\frac{\kappa_j-1}{2}}\int_{\mathbb {R}}e^{i\lambda(x-x^{\prime})}E_{\beta}\left(-\mu t^{\beta}|\lambda|^{\alpha}\right)E_{\beta}\left(-\mu (t')^\beta|\lambda|^{\alpha}\right)d\lambda\nonumber\\
     &=\sum_{j=1}^{n}2c_1(\kappa_j)A_{j}K_{\frac{\kappa_j-1}{2}}(|w_j|)|w_j|^{\frac{\kappa_j-1}{2}}\int_{0}^{
  +\infty}\cos\lambda(x-x')E_{\beta}(-\mu t^{\beta}|\lambda|^{\alpha})E_{\beta}(-\mu (t')^{\beta}|\lambda|^{\alpha})d\lambda\nonumber,
 \end{align}
 which completes the proof.
\end{proof}
\begin{theorem}\label{thm2ch5}
Let $u(t,x)$, $t > 0$, $x \in \mathbb{R}$, denote the random field defined by equations~{\rm (\ref{modch5})} and~{\rm(\ref{modCch5})} with the random initial condition $\xi(x),$ which covariance function satisfies Assumption~~{\rm \ref{Asumptionmain}} with $A_0\neq0$. Then, for $\alpha>\kappa_0/2$, if $\varepsilon \to 0$, the finite-dimentional distributions of the random fields
\[
U_\varepsilon(t,x) = \varepsilon^{-\frac{\kappa_0\beta}{2\alpha}}u\left( \frac{t}{\varepsilon}, \frac{x}{\varepsilon^{\beta/\alpha}} \right), \, t>0, x\in \mathbb{R},
\]
converge weakly to the finite-dimensional distributions of the zero-mean Gaussian field
\begin{equation}\label{Gzeromodch5}
    U_0(t,x) = \sqrt{c_{2}(\kappa_0)A_{0}}\int_{\mathbb{R}} \frac{e^{i \lambda x} E_\beta(-\mu|\lambda|^\alpha t^\beta)}{|\lambda|^{(1-\kappa_0)/2}}W(d\lambda), \,
\end{equation}
which has the covariance function
\begin{align}\label{covthmch5}
     {\rm Cov}(U_0(t,x), U_0(t', x')) =& 2A_{0}c_{2}(\kappa_0) \int_{0}^{+\infty} \frac{\cos\left( \lambda(x - x') \right)}{|\lambda|^{1 - \kappa_0}} E_\beta(-\mu|\lambda|^{\alpha} t^\beta)\nonumber\\&\times E_\beta(-\mu|\lambda|^{\alpha} (t')^\beta) \,d\lambda.
\end{align}

\end{theorem}
\begin{proof}
Using the same change of variables
\(
\lambda = \tilde{\lambda} \varepsilon^{\beta/\alpha},
\) as in the proof of Theorem~\ref{the1ch5}
one obtains
\begin{align}\label{firstch5}U_{\varepsilon}(t,x)
\overset{d}{=}&\;\sqrt{A_{0}c_2(\kappa_0)}
\int_{\mathbb{R}}
\frac{e^{i \tilde{\lambda} x} }{|\tilde{\lambda}|^{\frac{1-\kappa_0}{2}}}
E_{\beta}\!\left(
   -\mu \left|\tilde{\lambda}\,\varepsilon^{\beta/ \alpha}\right|^{\alpha}
   \Big(1 + \big|\tilde{\lambda}\,\varepsilon^{\beta/ \alpha}\big|^{2}\Big)^{\gamma/2}
   \left(\frac{t}{\varepsilon}\right)^{\beta}
\right) \nonumber \\
&\times
\Bigg(\sum_{\substack{j = -n \\ j \neq 0}}^{n}
\frac{A_{j}c_2(\kappa_j)}{2A_{0}c_{2}(\kappa_0)}
|\tilde{\lambda}|^{1-\kappa_0}\,\varepsilon^{(1-\kappa_0)\frac{\beta}{\alpha}}
   \frac{1-\theta_{\kappa_j}(|\tilde{\lambda}\,\varepsilon^{\beta/\alpha} + w_j|)}
        {\big|\tilde{\lambda}\,\varepsilon^{\beta/\alpha} + w_j\big|^{\,1-\kappa_j}}
 \nonumber \\
&
+ \Big(1 - \theta_{\kappa_0}\!\big(|\tilde{\lambda}\,\varepsilon^{\beta/\alpha}|\big)\Big)
\Bigg)^{1/2}
\, W(d\tilde{\lambda}).
\end{align}
It follows from identities (\ref{Gzeromodch5}) and (\ref{firstch5}) that
\begin{align*}
R(t,x)&
= \mathbb{E} \left( U_\varepsilon(t,x) - U_0(t,x) \right)^2 \\
&= \mathbb{E} \Bigg( \sqrt{A_{0}c_{2}(\kappa_{0})}
\int_{\mathbb{R}} \frac{e^{i \lambda x}}{|\lambda|^{\frac{1-\kappa_0}{2}}}
E_\beta \!\left(
-\mu |\lambda|^\alpha
\left(1 + |\lambda \varepsilon^{\beta/\alpha}|^2 \right)^{\gamma/2} t^\beta
\right) \\
& \times \Bigg( \sum_{\substack{j = -n \\ j \neq 0}}^{n}\frac{A_{j}c_2(\kappa_j)}{2A_{0}c_{2}(\kappa_0)}
|\lambda|^{1-\kappa_0}\,\varepsilon^{(1-\kappa_0)\frac{\beta}{\alpha}}
\frac{1 - \theta_{\kappa_j}(|\lambda \varepsilon^{\beta/\alpha}+w_j|)}
     {|\lambda \varepsilon^{\beta/\alpha} + w_j|^{1 - \kappa_j}}
 \\
&+
\left( 1 - \theta_{\kappa_0}\!\left( |\lambda \varepsilon^{\beta/\alpha}| \right) \right)
\Bigg)^{1/2} W(d\lambda)
- \sqrt{A_{0}c_2(\kappa_0)} \int_{\mathbb{R}}\frac{e^{i \lambda x}}{|\lambda|^{\frac{1-\kappa_0}{2}}}
E_\beta\!\left(-\mu |\lambda|^\alpha t^\beta\right)
W(d\lambda)
\Bigg)^2\\
&= A_{0}c_2(\kappa_0)
\int_{\mathbb{R}}
\left(
\frac{E_\beta(-\mu |\lambda|^\alpha t^\beta)}{|\lambda|^{\frac{1 - \kappa_0}{2}}}
\left(
\tilde{Q}_{\varepsilon}(\lambda) - 1
\right)
\right)^2 d\lambda,
\end{align*}
where \begin{align}\label{Qch51}
   & \tilde{Q}_{\varepsilon}(\lambda) =
\frac{ E_\beta\!\left(-\mu |\lambda|^\alpha\left(1 + |\lambda \varepsilon^{\beta/\alpha}|^2\right)^{\gamma/2} t^\beta\right)}
     {E_\beta\!\left(-\mu |\lambda|^\alpha t^\beta\right)}\Bigg(
\frac{|\lambda|^{1-\kappa_0}\,\varepsilon^{(1-\kappa_0)\frac{\beta}{\alpha}}}{2} \nonumber\\
&\times
\sum_{\substack{j = -n \\ j \neq 0}}^{n}\frac{A_{j}c_2(\kappa_j)}{A_{0}c_2(\kappa_0)}
\frac{1 - \theta_{\kappa_j}(|\lambda \varepsilon^{\beta /\alpha}+w_j|)}
     {|\lambda \varepsilon^{\beta/\alpha} + w_j|^{1-\kappa_j}} +(1 - \theta_{\kappa_0}(|\lambda \varepsilon^{\beta/\alpha}|))
\Bigg)^{1/2}.
\end{align}
Since $\theta_{\kappa_j}(\cdot)$ are continuous at $w_j,\, j=1,...,n,$ and
$\lim_{\varepsilon \to 0} \theta_{\kappa_0}(\varepsilon)=0, $ it follows that, pointwise for all $\lambda\in\mathbb{R}$, it holds $\lim_{\varepsilon \to 0} \tilde{Q}_\varepsilon(\lambda)=1$.
Then, similar to the proof of Theorem \ref{the1ch5}, the generalized Lebesgue dominated convergence theorem can be applied.

To justify its conditions, one needs to show that the limit
\begin{align}
&\lim_{\varepsilon\to 0}A_{0}c_{2}(\kappa_0)\int_{\mathbb{R}}\frac{E_\beta^{2}(-\mu |\lambda|^\alpha t^\beta)}{|\lambda|^{1 - \kappa_0}}
 \left( {\tilde{Q}^{2}_{\varepsilon}(\lambda)}+1 \right) d\lambda\label{1ch5}\\&=A_{0}c_{2}(\kappa_0)\int_{\mathbb{R}}\frac{E_\beta^{2}(-\mu |\lambda|^\alpha t^\beta)}{|\lambda|^{1 - \kappa_{0}}}\lim_{\varepsilon\to 0}\left(  {\tilde{Q}^{2}_{\varepsilon}(\lambda)}+1 \right)
 d\lambda\nonumber\\&=2A_{0}c_{2}(\kappa_0)\int_{\mathbb{R}}\frac{E_\beta^{2}(-\mu |\lambda|^\alpha t^\beta)}{|\lambda|^{1 - \kappa_{0}}}
 d\lambda<\infty.\label{limit ch5}
\end{align}
 To show that the integral in \eqref{limit ch5} is finite for $\kappa_{0}\in(0,1)$, let us change the variable as  $\tilde{\lambda}=\mu t^{\beta}\lambda^{\alpha}$. Then, $\lambda=\left(\tilde{\lambda}/(\mu t^{\beta})\right)^{\frac{1}{\alpha}},$  $d\lambda=\tilde{\lambda}^{{1/\alpha}-1}/(\alpha(\mu t^{\beta})^{1/\alpha})d\tilde{\lambda}.$
By splitting the integral into two parts for some $A > 0$ and using the upper bound from~\eqref{upperch5}, it follows that the expression in \eqref{limit ch5} is equal
\begin{align*}
  &4A_{0}c_{2}(\kappa_0)\alpha^{-1} (\mu t^{\beta})^{-\kappa_{0}/\alpha} \int_0^{\infty} E_\beta^{2}(-\tilde{\lambda}) \tilde{\lambda}^{-\frac{1 - \kappa_{0}}{\alpha}} \tilde{\lambda}^{\frac{1}{\alpha} - 1} \, d\tilde{\lambda}\nonumber\\&
\leq4A_{0}c_{2}(\kappa_0)\alpha^{-1}  (\mu t^{\beta})^{-\kappa_{0}/\alpha} \left( \int_0^{A} \frac{\tilde{\lambda}^{\frac{\kappa_{0} }{\alpha}-1}}{\left(1 +\tilde{\lambda}/\Gamma(1+\beta)\right)^2}d\tilde{\lambda} + \int_{A}^{+\infty} \frac{\tilde{\lambda}^{\frac{\kappa_{0} }{\alpha}-1}}{\left(1 +\tilde{\lambda}/\Gamma(1+\beta)\right)^2}d\tilde{\lambda} \right)\nonumber\\&\leq4A_{0}c_{2}(\kappa_0)\alpha^{-1}  (\mu t^{\beta})^{-\kappa_{0}/\alpha} \left( \int_0^{A} \frac{d\tilde{\lambda}}{\tilde{\lambda}^{1-\frac{\kappa_{0}}{\alpha}}} + \int_A^{+\infty} \frac{\Gamma^{2}(1+\beta)}{\tilde{\lambda}^{3-\frac{\kappa_{0}}{\alpha}}}d\tilde{\lambda} \right).
\end{align*}
The integral above is finite if $\alpha>\kappa_{0}/2$.
It follows from the upper and lower bounds in~(\ref{upperch5}) that this condition is sufficient and necessary for the boundedness of the integral in~(\ref{limit ch5}).

Now, let's consider the integral in (\ref{1ch5})
\begin{align}
&\int_{\mathbb{R}}
  \frac{E_\beta^{2}(-\mu |\lambda|^\alpha t^\beta)}{|\lambda|^{1 - \kappa_{0}}}
  \left(  \tilde{Q}^2_{\varepsilon}(\lambda)  + 1 \right)\, d\lambda
   =\frac{\varepsilon^{(1-\kappa_0)\frac{\beta}{\alpha}}}{2}\int_{\mathbb{R}}E_\beta^{2}(-\mu |\lambda|^\alpha(1 + |\lambda\varepsilon^{\beta/\alpha}|^2)^{\gamma/2} t^\beta)\nonumber\\ &\quad\times\sum_{\substack{j = -n \\ j \neq 0}}^{n}\frac{A_{j}c_2(\kappa_j)}{A_{0}c_2(\kappa_0)}\frac{1 - \theta_{\kappa_j}(|\lambda \varepsilon^{\beta/\alpha}+w_j|) } {|\lambda \varepsilon^{{\beta}/{\alpha}}+w_j|^{1-\kappa_j}} d\lambda +\int_{\mathbb{R}}\frac{ E_\beta^{2}(-\mu |\lambda|^\alpha(1 + |\lambda\varepsilon^{\beta /\alpha}|^2)^{\gamma/2} t^\beta)}
{|\lambda|^{1 - \kappa_{0}}}\nonumber\\
&\quad\times(1 - \theta_{\kappa_0}(|\lambda \varepsilon^{\beta /\alpha}|)) d\lambda+\int_{\mathbb{R}}
   \frac{E_\beta^{2}(-\mu |\lambda|^\alpha t^\beta)}{|\lambda|^{1 - \kappa_{0}}}d\lambda\nonumber
   \end{align}
   \begin{equation} \label{lonch5}
   =\frac{\varepsilon^{(1-\kappa_0)\frac{\beta}{\alpha}}}{2}\sum_{\substack{j = -n \\ j \neq 0}}^{n}\frac{A_{j}c_2(\kappa_j)}{A_{0}c_2(\kappa_0)}I_{1}(\varepsilon^{\beta/\alpha})+I_{2}(\varepsilon^{\beta/\alpha})+I_{3}.
\end{equation} Note that,
$\lim_{\varepsilon\to 0}{\varepsilon^{(1-\kappa_0)\frac{\beta}{\alpha}}}\sum_{\substack{j = -n \\ j \neq 0}}^{n}\frac{A_{j}c_2(\kappa_j)}{A_{0}c_2(\kappa_0)}(I_{1}(\varepsilon^{\beta/\alpha})+I_{2}(\varepsilon^{\beta/\alpha}))=0$, as $\varepsilon^{(1-\kappa_0)\frac{\beta}{\alpha}}\to 0$ and the integrals $I_{1}$ and $I_{2}$ are bounded. The integral $I_{3}$ is finite as was shown on \eqref{limit ch5}.
It follows from $\lim_{\varepsilon\to 0}\theta_{0}(\varepsilon)=0, \,\theta_{0}(\lambda)\in[0,1]$, and the dominated convergence theorem that the integral $I_{2}(\varepsilon^{\beta/\alpha})\to I_{3}$, when $\varepsilon\to 0$.  Hence, the result in \eqref{1ch5}-\eqref{limit ch5} follows from \eqref{lonch5}. By the generalized dominated convergence theorem  $\lim_{\varepsilon \to 0}E{\vert U_{\varepsilon}(t,x)-U_{0}(t,x)\vert}^{2}=0,$ which implies the convergence of finite-dimensional distributions.

Finally, to obtain equation \eqref{covthmch5} note that
 \begin{align}&{\rm Cov}(U_{0}(t,x),U_{0}(t',x'))=\mathbb E U_{0}(t,x)U_{0}(t^{\prime},x^{\prime})\nonumber\\&=\mathbb E\left(A_{0}c_{2}(\kappa_0)\int_{\mathbb R^2} \frac{e^{i\lambda_{1}x}}{|\lambda_1|^{\frac{1-\kappa_{0}}{2}}} E_{\beta}\left(-\mu t^{\beta}|\lambda_{1}|^{\alpha}\right) \frac{e^{i\lambda_{2}x^{\prime}}}{{|\lambda_2|^\frac{1-\kappa_{0}}{2}}} E_{\beta}\left(-\mu (t')^{\beta}|\lambda_{2}|^{\alpha}\right)W(d\lambda_{1})W(d\lambda_{2})\right)\nonumber\\
     &=2A_{0}c_{2}(\kappa_0)\int_{0}^{
  +\infty}\frac{\cos(\lambda(x-x'))}{|\lambda|^{1-\kappa_0}}E_{\beta}(-\mu t^{\beta}|\lambda|^{\alpha})E_{\beta}(-\mu (t')^{\beta}|\lambda|^{\alpha})d\lambda\nonumber,
 \end{align}
 which completes the proof.
\end{proof}
\begin{Remark}
Similar to the paper \cite{alghamdi2025multiscalingasymptoticbehaviorsolutions}, the singularity at the origin "dominates" all other singularities and completely determines the asymptotic behavior.
\end{Remark}

Notice that the limit random fields in Theorems~\ref{the1ch5} and~\ref{thm2ch5} are stationary in the coordinate~$x$ and temporally nonstationary in the coordinate~$t$.

\begin{example}
This example demonstrates how the covariance function of the limit field
$U_{0}(t,x)$, given in Theorems~\ref{the1ch5} and \ref{thm2ch5}, depends on time,
space, and the parameters. Without loss of generality, we assume that in
(\ref{cov of Bessel ch5}) and (\ref{covthmch5}) the constants
$\sum_{j=1}^{n} 2 c_{1}(\kappa_{j}) A_{j}
K_{\frac{\kappa_{j}-1}{2}}(|w_{j}|)\, |w_{j}|^{\frac{\kappa_{j}-1}{2}}$
and $2A_{0}c_{2}(\kappa_{0})$
are equal to $1$. Since the covariance
$\mathrm{Cov}\!\left(U_{0}(t,x),\, U_{0}(t',x')\right)$
depends on four variables $t$, $t'$, $x$, and $x'$, as well as on the parameters
$\alpha$, $\beta$, and $\mu$, we illustrate its behavior using two-dimensional
plots in which some variables are held fixed. We set $\mu = \alpha = 1$, $\beta=1/2$, and consider the cases of
$\kappa_{0} =0.2,\, 0.5,$ and $0.7$.

For the case $\beta = {1}/{2}$, we use the identity
$E_{1/2}(z) = e^{z^{2}} \operatorname{erfc}(-z),$
where the Gauss error function is defined by
$\operatorname{erf}(z) = \frac{2}{\sqrt{\pi}} \int_{0}^{z} e^{-t^{2}}\, dt.$

For $t = t' = 1$, Figure~\ref{covfigch5} shows the covariance as a function of
the spatial increment $x-x' \in [-15, 15]$. The covariance is a symmetric function. The figure
also demonstrates that smaller values of $\kappa_{0}$ yield larger
covariances and slower spatial decay. In contrast, larger values of $\kappa_{0}$ produce smaller peaks and
more rapid decay as $|h|$ increases.

 For $x-x'=1$ and $t=1$, Figure \ref{covmfigch5} presents the covariance decreases monotonically
as $t'$ increases. As before,
smaller values of $\kappa_{0}$ yield larger covariance values and slower decay,
while larger values of $\kappa_{0}$ result in smaller values and more rapid
decay as $t'$ grows.
    \begin{figure}[tb]
        \centering
        \includegraphics[width=0.9\linewidth]{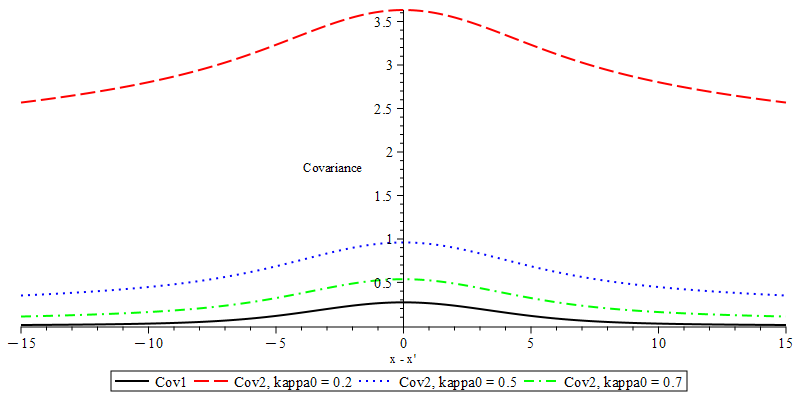}
        \caption{Covariance as a function of $x-x'$ with $t=t'=1$}
        \label{covfigch5}
    \end{figure}
    \begin{figure}
        \centering
        \includegraphics[width=0.9\linewidth]{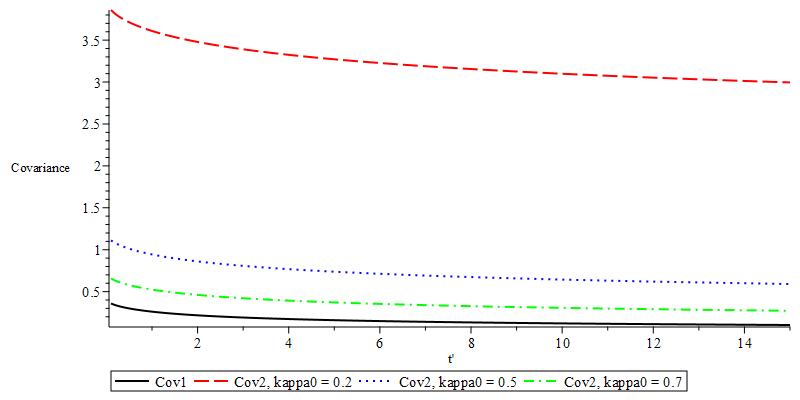}
        \caption{Covariance as a function of $t'$ with fix $x-x'=1$ and $t=1$}
        \label{covmfigch5}
    \end{figure}
\end{example}
\section{Multiscaling limit theorem in the case of regular variation} \label{sec4ch5}
This section presents a generalization of Theorem \ref{thm2ch5} for the case of regular varying asymptotics. To formulate the corresponding results, we need the following definitions and properties of slowly varying functions.
\begin{definition} \cite{seneta2006regularly}
A function $L(\cdot)$ is slowly varying at infinity if it is real-valued, positive, and measurable on $[A,\infty)$ for some $A>0$, and satisfies the following condition for every $\lambda >0$:
\begin{equation*}
    \lim_{x\to \infty} \frac{L(\lambda x)}{L(x)}=1.
\end{equation*}
\end{definition}
\begin{definition}\cite{Bingham_Goldie_Teugels}
A measurable function $f(\cdot)>0$ is regularly varying with index $\rho$ if it satisfies the following asymptotic property:
\[
\lim_{x\to \infty}\frac{f(\lambda x)}{f(x)}=\lambda^{\rho}, \quad \lambda>0,
\]
for some finite constant $\rho \in \mathbb{R}$.
\end{definition}
\begin{Proposition}\cite[Proposition 1.3.6]{Bingham_Goldie_Teugels}\label{propch5}
    If $L$ varies slowly and $\alpha>0,$
    \[x^{\alpha}L(x)\to\infty, \quad x^{-\alpha}L(x)\to 0,\quad (x\to \infty). \]
\end{Proposition}
We will use the next generalization of Assumplion~\ref{Asumptionmain} and the spectral density (\ref{spectral density Ch5}).
\begin{assumption}\label{assumption 2ch5}
    The spectral density of the random process $\xi(x)$ is given by \[f^{(L)}(\lambda)=\sum_{j=-n}^{n} \frac{c_{2}(\kappa_j)}{2}A_{j}\frac{1-\theta_{\kappa_j}(|\lambda+w_{j}|)}{|\lambda+w_j|^{1-\kappa_j}}L_{j}\left(\frac{1}{|\lambda+w_j|}\right),\]
    where $L_{j}(\cdot),\, j=-n,...,n,$ are slowly varying at infinity functions, which are bounded on each finite integral. All other values are the same as in Assumption \ref{Asumptionmain}.
\end{assumption}
As can be seeing from this assumption, the spectral density has power-type singularities at the locations $w_{j}, \, j=-n,...,n,$ with the corresponding power-laws parameters $1-\kappa_{j}$. The deviations of the spectral density from this power-laws are given by slowly varying functions $L_{j}(\cdot)$. The boundedness of functions $L_{j}(\cdot)$ on each finite interval is required to not introduce other singularities in the spectrum.
\begin{theorem}
    Let $u(t,x)$, $t > 0$, $x \in \mathbb{R}$, be the random field defined by equations~{\rm (\ref{modch5})} and~{\rm(\ref{modCch5})} with the random initial condition $\xi(x),$ which covariance function satisfies Assumption~~{\rm \ref{assumption 2ch5}} with $A_0\neq0$. Then, for $\alpha>\kappa_0/2$, if $\varepsilon \to 0$, the finite-dimentional distributions of the random fields
\[
U^{(L)}_\varepsilon(t,x) = \frac{\varepsilon^{-\frac{\kappa_0\beta}{2\alpha}}}{L_{0}\left({\varepsilon^{-\beta/\alpha}}\right)}u\left( \frac{t}{\varepsilon}, \frac{x}{\varepsilon^{\beta/\alpha}} \right), \, t>0, \ x\in \mathbb{R},
\]
converge weakly to the finite-dimensional distributions of the zero-mean Gaussian field
\begin{equation*}
    U_0(t,x) = \sqrt{c_{2}(\kappa_0)A_{0}}\int_{\mathbb{R}} \frac{e^{i \lambda x} E_\beta(-\mu|\lambda|^\alpha t^\beta)}{|\lambda|^{(1-\kappa_0)/2}}W(d\lambda), \,
\end{equation*}
\end{theorem} \begin{proof}
    Using the same $U_{\varepsilon}(t,x)$ as in Theorem \ref{thm2ch5} one obtains that \begin{align}\label{slowch5}
        \mathbb{E}\left(U^{(L)}_{\varepsilon}(t,x)-U_{0}(t,x)\right)^{2}&\leq2\mathbb{E}\left(U^{(L)}_{\varepsilon}(t,x)-U_{\varepsilon}(t,x)\right)^{2}\nonumber\\&+2\mathbb{E}\left(U_{\varepsilon}(t,x)-U_{0}(t,x)\right)^{2}.
    \end{align} As it was demonstrated in the proof of Theorem \ref{thm2ch5}, the second summand in \eqref{slowch5} approches $0$, when $\varepsilon\to 0$.

    Let us consider the first term in \eqref{slowch5}:\begin{align}
\mathbb{E}&\!\left( U^{(L)}_{\varepsilon}(t,x) - U_{0}(t,x) \right)^{2}
= \int_{\mathbb{R}}
   \frac{
      E^{2}_{\beta}\!\left(
        -\mu |\lambda|^{\alpha}
        \left( 1 + |\lambda \varepsilon^{\beta/\alpha}|^{2} \right)^{\gamma/2}
        t^{\beta}
      \right)
   }{|\lambda|^{1-\kappa_{0}}}
   \nonumber \\
& \times
   \Biggl(
     \Biggl(
       \sum_{j=-n}^{n}
       \frac{A_{j} c_{2}(\kappa_j)}{2}
       \frac{1 - \theta_{\kappa_j}\!\left( |\lambda\varepsilon^{\beta/\alpha} + w_{j}| \right)}
            {|\lambda\varepsilon^{\beta/\alpha} + w_{j}|^{1 - \kappa_{j}}}
       |\lambda|^{1-\kappa_{0}}
       \varepsilon^{(1-\kappa_{0}){\beta/\alpha}}
     \Biggl)^{1/2} -\left(
       \sum_{j=-n}^{n}
       \frac{A_{j} c_{2}(\kappa_j)}{2}\right.
     \nonumber \end{align}
    \begin{equation}\label{the diff ch5} \left.\left. \times
       \frac{1 - \theta_{\kappa_j}\!\left( |\lambda\varepsilon^{\beta/\alpha} + w_{j}| \right)}
            {|\lambda\varepsilon^{\beta/\alpha} + w_{j}|^{1 - \kappa_{j}}}
       |\lambda|^{1-\kappa_{0}}
       \varepsilon^{(1-\kappa_{0}){\beta/\alpha}}
       \frac{
          L_{j}\!\left( \dfrac{1}{|\lambda\varepsilon^{\beta/\alpha} + w_{j}|} \right)
       }{
          L_{0}\!\left( \dfrac{1}{\varepsilon^{\beta/\alpha}} \right)
       }
     \right)^{1/2}
   \right)^{2}
   \, d\lambda .
\end{equation}
Now, we will use the inequalities
\begin{align*}
   & \left(\Bigg(\sum_{j=-n}^{n}a_{j}\Bigg)^{1/2}-\Bigg(\sum_{j=-n}^{n}b_{j}\Bigg)^{1/2}\right)^{2}\leq \left|\Bigg(\sum_{j=-n}^{n}a_{j}\Bigg)^{1/2}-\Bigg(\sum_{j=-n}^{n}b_{j}\Bigg)^{1/2}\right|\nonumber\\&\times\left(\Bigg(\sum_{j=-n}^{n}a_{j}\Bigg)^{1/2}+\Bigg(\sum_{j=-n}^{n}b_{j}\Bigg)^{1/2}\right)\leq|\sum_{j=-n}^{n}(a_{j}-b_{j})|\leq\sum_{j=-n}^{n}|a_j-b_j|,
\end{align*} where $\{a_j\}$ and $\{b_j\}$ are sets of non-negative numbers.

By applying it to \eqref{the diff ch5} one obtains
\begin{align}\label{diffslowch5}
  \mathbb{E}&\left( U^{(L)}_{\varepsilon}(t,x) - U_{0}(t,x) \right)^{2}
= \int_{\mathbb{R}}\frac{E^{2}_{\beta}\left(-\mu |\lambda|^{\alpha}\left( 1 + |\lambda \varepsilon^{\beta/\alpha}|^{2} \right)^{\gamma/2} t^{\beta} \right) }{|\lambda|^{1-\kappa_{0}}}|\lambda|^{1-\kappa_{0}}\varepsilon^{(1-\kappa_{0})^{\beta/\alpha}}
   \nonumber\\& \times \sum_{j=-n}^{n} \frac{A_{j} c_{2}(\kappa_j)}{2}\frac{1 - \theta_{\kappa_j}\!\left( |\lambda\varepsilon^{\beta/\alpha} + w_{j}| \right)} {|\lambda\varepsilon^{\beta/\alpha} + w_{j}|^{1 - \kappa_{j}}}\left\lvert 1- \frac{ L_{j}\left( \dfrac{1}{|\lambda\varepsilon^{\beta/\alpha} + w_{j}|} \right) }{ L_{0}\left( \dfrac{1}{\varepsilon^{\beta/\alpha}} \right) }\right\lvert \, d\lambda .
\end{align} The summands with $j\neq 0$ can be bounded by
\begin{align}\label{diffslowly}
    &\varepsilon^{(1-\kappa_{0})^{\beta/\alpha}}
    \int_{\mathbb{R}}
        E^{2}_{\beta}\left(
            -\mu |\lambda|^{\alpha}
            \left( 1 + |\lambda \varepsilon^{\beta/\alpha}|^{2} \right)^{\gamma/2}
            t^{\beta}
        \right)
        \frac{1 - \theta_{\kappa_{j}}(|\lambda \varepsilon^{\beta/\alpha} + w_{j}|)}
             {|\lambda \varepsilon^{\beta/\alpha} + w_{j}|^{\,1-\kappa_{j}}}
    \, d\lambda
    \nonumber \\
    & +
    \varepsilon^{(1-\kappa_{0})^{\beta/\alpha}}
    \int_{\mathbb{R}}
        E^{2}_{\beta}\left(
            -\mu |\lambda|^{\alpha}
            \left( 1 + |\lambda \varepsilon^{\beta/\alpha}|^{2} \right)^{\gamma/2}
            t^{\beta}
        \right)
        \frac{1 - \theta_{\kappa_{j}}(|\lambda \varepsilon^{\beta/\alpha} + w_{j}|)}
             {|\lambda \varepsilon^{\beta/\alpha} + w_{j}|^{\,1-\kappa_{j}}}
       \nonumber\\& \times\frac{
            L_{j}\left( \dfrac{1}{|\lambda\varepsilon^{\beta/\alpha} + w_{j}|} \right)
        }{
            L_{0}\left( \dfrac{1}{\varepsilon^{\beta/\alpha}} \right)
        }
    d\lambda.
\end{align}
By repeating the same steps as in the proof of Theorem \ref{thm2ch5}, starting from \eqref{Qch51}, one can see that the first integral in \eqref{diffslowly} vanishes, when $\varepsilon\to 0.$

By Proposition \ref{propch5} and Assumption \ref{assumption 2ch5}, for any $\nu>0$, there are positive constants $C,\, C_{1}$ and $\varepsilon_{0}$ that it holds \[\frac{
            L_{j}\left( \dfrac{1}{|\lambda\varepsilon^{\beta/\alpha} + w_{j}|} \right)
        }{
            L_{0}\left( \dfrac{1}{\varepsilon^{\beta/\alpha}} \right)
        }\leq C_{1}\frac{\max \left(C,|\lambda\varepsilon^{\beta/\alpha}+w_{j}|^{-\nu}\right)}{\varepsilon^{\nu}},\]  for all $\lambda$, such that $\lambda\varepsilon^{\beta/\alpha}+w_j\neq 0$ and $\varepsilon\geq\varepsilon_{0}$.

        Therefore, the second integral in \eqref{diffslowly} is bounded by the integral of the same form as the first integral in \eqref{diffslowly} but with the multiplier $\varepsilon^{(1-\kappa_{0}){\beta/\alpha}-\nu}$ and with the denominator power $1-\kappa_{j}+\nu$. As $\nu>0$ can be chosen arbitrarily small, we conclude that the second integral in \eqref{diffslowly} vanishes, when $\varepsilon\to 0$.

        Now, let us consider the summand corresponding to $j=0$ in \eqref{diffslowch5}. Since the Mittag-Leffer function $E_{\beta}(\cdot)$ is a monotonically decreasing function of negative arguments and $|\theta_{\kappa_{0}}(\lambda)|\leq1$ one obtains the next upper bound
        \begin{align}\label{onesqch5}
             &\varepsilon^{(1-\kappa_{0}){\beta/\alpha}}
    \int_{\mathbb{R}}
        E^{2}_{\beta}\left(
            -\mu |\lambda|^{\alpha}
            \left( 1 + |\lambda \varepsilon^{\beta/\alpha}|^{2} \right)^{\gamma/2}
            t^{\beta}
        \right)\frac{A_{0}c_{2}(\kappa_0)}{2}
        \frac{1 - \theta_{\kappa_{0}}(|\lambda \varepsilon^{\beta/\alpha}|)}
             {|\lambda \varepsilon^{\beta/\alpha}|^{\,1-\kappa_{j}}}
       \nonumber\\& \times\left|1-\frac{
            L_{0}\left( \frac{1}{|\lambda\varepsilon^{\beta/\alpha}|} \right)
        }{
            L_{0}\left( \dfrac{1}{\varepsilon^{\beta/\alpha}} \right)
        }\right|d\lambda\leq A_{0}c_{2}(\kappa_0)\int_{0}^{+\infty}\frac{E_{\beta}^{2}(-\mu \lambda^{\alpha}t^{\beta})}{\lambda^{1-\kappa_{0}}}\left|1-\frac{
            L_{0}\left( \dfrac{1}{\lambda\varepsilon^{\beta/\alpha}} \right)
        }{
            L_{0}\left( \dfrac{1}{\varepsilon^{\beta/\alpha}} \right)
        }\right|d\lambda.
        \end{align}

        Note that for $\nu\in \mathbb{R}$, such that $|\nu|<\min(1-\kappa_{0},\kappa_{0})$, it holds
        \begin{equation}\label{secondsqch5}
           \int_{0}^{\infty} \frac{E^{2}_{\beta}(-\mu \lambda^{\alpha}t^{\beta})}{\lambda^{1-\kappa_{0}-\nu}}d\lambda<+\infty.
        \end{equation}
        Then using Theorems 2.6 and 2.7 in \cite{seneta2006regularly} one obtains that the upper bound in \eqref{onesqch5} approaches $0$, when $\varepsilon\to 0$. The condition \eqref{secondsqch5} guarantees the finiteness of the integrals in the conditions of Theorems 2.6 and 2.7 \cite{seneta2006regularly}. Note that, instead of the absolute value $\left|1-\frac{
            L_{0}\left({1}/(\lambda\varepsilon^{\beta/\alpha}) \right)
        }{
            L_{0}\left({\varepsilon^{-\beta/\alpha}} \right)
        }\right|$ these theorems use the difference itself. However, in their proofs, the difference appears only in the application of the uniform convergence theorem on the finite interval $[\alpha,\beta], \, 0<\alpha<\beta<+\infty$. So, the results obviously hold true for the absolute value too. Hence, $\mathbb{E}\left( U^{(L)}_{\varepsilon}(t,x) - U_{0}(t,x) \right)^{2}\to 0$,  when $ \varepsilon\to 0$, which completes the proof.
\end{proof}
\section{Conclusion}\label{conch5}
This paper analyzes fractional Riesz–Bessel equations with random initial conditions exhibiting both classical long-range dependence and cyclic long-range dependence. It establishes the convergence of rescaled solutions to spatio-temporal Gaussian random fields and proves multiscaling limit theorems for regularly varying asymptotics.
It would be interesting to generalize the results and further develop the asymptotic approach in the regular-varying cases. Specifically, to study the case with $A_{0}=0$, when the classical long-memory component is not present. Also, it would be interesting to study the case of more general classes of regularly varying functions (see, \cite{LeonenkoOlenko2013, Olenko2005I, Olenko2007OR_II}). Finally, generalizations of the results to multidimensional spatial settings and subordinated initial conditions remain challenging open problems.
\section*{Acknowledgments}
This research was supported by the Australian Research Council's Discovery Projects funding scheme (project number DP220101680).  A.~Olenko was also partially supported by La Trobe University's SCEMS CaRE and Beyond grant.
\bibliographystyle{amsplain}
\bibliography{Bibliography}

\end{document}